\documentclass[11pt, a4paper, reqno]{amsart}
\usepackage[top=35truemm,bottom=35truemm,left=35truemm,right=35truemm]{geometry}
\usepackage{setspace}
\usepackage{amsmath, amsfonts,amsthm,amssymb,mathrsfs,amscd}
\usepackage{bm}
\usepackage{enumerate}
\usepackage{graphicx}
\usepackage{ascmac}
\usepackage[arrow,matrix]{xy}
\usepackage{color}
\usepackage{array}
\usepackage{pifont}
\usepackage{longtable}
\usepackage{here}
\usepackage{braket}
\usepackage{mathtools}
\usepackage{tikz}

\newtheorem{theo}{Theorem}[section]
\newtheorem*{theo*}{Theorem}

\newtheorem{prop}{Proposition}[section]

\newtheorem{lem}{Lemma}[section]

\newtheorem{introthm}{Theorem}

\providecommand{\abs}[1]{\left\lvert#1\right\rvert}

\providecommand{\presentation}[2]{\left\langle#1 \middle| #2\right\rangle}

\DeclareMathOperator{\Vol}{Vol}

\DeclareMathOperator{\Mag}{Mag}

\makeatletter
\@addtoreset{equation}{section}

\makeatother

\makeatletter
\newcommand{\vb}{\@ifstar{\vb@nostar}{\vb@star}}

\newcommand{\vb@nostar}[1]{\bm{#1}}

\newcommand{\vb@star}[1]{%
  \accentset{\smash{\rightharpoonup}}{\bm{#1}}%
}
\makeatother

\def\disp{\displaystyle}

\def\Z{\mathbb{Z}}

\def\N{\mathbb{N}}

\def\R{\mathbb{R}}

\def\({\left(}
\def\){\right)}
\def\<{\langle}
\def\>{\rangle}

\begin{document}

\title[ ]{Magnitude of homogeneous Moran sets in the unit interval}
\author{Ryo Matsuda}
\address{Department of Mathematical Sciences, College of Science and Engineering, Ritsumeikan University, 1-1-1 Nojihigashi, Kusatsu, Shiga 525-8577, Japan}
\email{r-mat@fc.ritsumei.ac.jp}
\author{Tomoshige Yukita}
\address{Liberal Arts Education Center, Ashikaga University, 268-1 Omaecho, Ashikaga,Tochigi 326-8558, Japan}
\email{yukita.tomoshige@g.ashikaga.ac.jp}
\date{}
\thanks{}

\begin{abstract}
Magnitude, denoted by $\operatorname{Mag}(X)$, is a real-valued invariant of compact metric spaces whose large-scale growth reflects their geometry. Willerton showed that, for a compact homogeneous Riemannian manifold $X$, $\operatorname{Mag}(tX)$ grows like $t^{\dim X}$, with the volume of $X$ appearing in its leading asymptotic terms.
We study a homogeneous Moran Cantor set $E$ equipped with the Euclidean metric $d$ and with its coding ultrametric $d_u$,
writing $E_u=(E,d_u)$.
We prove that the upper and lower growth exponents of $\operatorname{Mag}(tE_u)$,
called the magnitude dimensions of $E_u$,
coincide respectively with the upper and lower Euclidean box dimensions of $E$.
In the self-similar case with constant contraction ratio $r$,
we obtain $\operatorname{Mag}(tE_u)=t^s/\widetilde{p}(\log t)+o(t^s)$ as $t\to\infty$,
where $s$ is the Hausdorff dimension of $E$ and $\widetilde{p}$ is a positive smooth function of period $-\log r$.
The harmonic mean of the leading coefficient $1/\widetilde{p}$ is $m\log m/((m-1)\Gamma(s+1))$,
giving a fractal analogue of Willerton's leading-order asymptotics with a log-periodic,
rather than constant, coefficient.
\end{abstract}

\maketitle
\section{Introduction}
Magnitude is a real-valued invariant introduced by Leinster for finite metric spaces as a metric analogue of the Euler characteristic~\cite{Leinster2013}, and later extended to broad classes of compact metric spaces, in particular by Meckes~\cite{Meckes2013}.
For a finite metric space,
magnitude is defined from the similarity matrix $\zeta_X=(e^{-d(x, y)})_{x.y\in X}$ and is considered as an effective number of distinguishable points at the scale determined by the metric.
Further details are given in Section~\ref{sec:magnitude definition}.

\smallskip
For a metric space $X$ and $t>0$, let $tX$ denote the same set with all distances multiplied by $t$.
The function $\Mag(tX)$ in $t\in \R_{>0}$ is called the magnitude function of $X$.
As $t$ increases, the space is viewed at finer and finer scales, so the growth of $\Mag(tX)$ may reflect the small-scale geometry of $X$.

\smallskip
The upper and lower growth exponents of $\Mag(tX)$ are defined by
\[
    \overline{\dim}_{\Mag}(X)
    :=
    \limsup_{t\to\infty}
    \frac{\log\Mag(tX)}{\log t},
    \quad
    \underline{\dim}_{\Mag}(X)
    :=
    \liminf_{t\to\infty}
    \frac{\log\Mag(tX)}{\log t}.
\]
These quantities are called the upper and lower magnitude dimensions of $X$, respectively.
Meckes established a fundamental connection between magnitude and box dimensions of compact subspaces of Euclidean space or compact ultrametric spaces~\cite[Corollaries 7.3 and 7.4]{Meckes2015} (see Theorem~\ref{theo:Meckes dimensions}).
In both results, however, the magnitude dimension and the box dimension are computed using the same metric.

\smallskip
Our first aim is to exhibit a connection between magnitude dimensions and box dimensions with respect to different metrics on the same underlying compact set.
Our example is a Cantor set arising from a homogeneous Moran construction, equipped with both its Euclidean metric and a coding ultrametric.
It suggests that, for suitable fractals, the magnitude dimensions of a coding ultrametric may recover the box dimensions of the same set in its original metric.

\smallskip
Let $E=E(m,\mathbf{r})\subset[0,1]$ be a homogeneous Moran set with fixed branching number $m\geq2$ and contraction ratios $\mathbf{r}=(r_k)_{k\geq1} \ (r_k<m^{-1})$ as defined in Section~\ref{sec:Moran set and ultrametric}.
The coding of the Moran construction identifies $E$ with the symbolic space $\{1,\dots,m\}^{\mathbb{N}}$ and induces a natural ultrametric $d_u$.
We write $E_u$ for the compact ultrametric space $(E, d_u)$.
We continue to write $E$ for the homogeneous Moran set equipped with the Euclidean metric.
By Brouwer's characterization of the Cantor space~\cite[Theorem 7.4, p.35]{Kechris1995},
both the compact metric spaces $E$ and $E_u$ are Cantor spaces.
Our first main result shows that the magnitude dimensions of $E_u$ coincide with the Euclidean box dimensions of $E$.

\begin{introthm}\label{introthm:dimensions}
Let $E=E(m,\mathbf{r})$ be the homogeneous Moran set constructed in Section~\ref{sec:Moran set and ultrametric},
and let $E_u=(E,d_u)$ be the compact ultrametric space induced by its coding. Then
\[
    \overline{\dim}_{\Mag}(E_u)
    =\overline{\dim}_B(E),
    \qquad
    \underline{\dim}_{\Mag}(E_u)
    =\underline{\dim}_B(E),
\]
where the box dimensions on the right-hand side are computed with respect to the Euclidean metric on $E$.
\end{introthm}

This differs from Meckes' same-metric comparison.
Here,
the magnitude dimension is computed using $d_u$,
while the box dimension is computed using the Euclidean metric.
The two metrics need not be bi-Lipschitz equivalent.
In fact,
we prove as Proposition~\ref{prop:difference of metrics} that they are bi-Lipschitz equivalent if and only if
\[
    \sup_{k\geq1}r_k<\frac{1}{m}.
\]
Theorem~\ref{introthm:dimensions} remains valid without this stronger assumption.
Hence the result is not simply a consequence of bi-Lipschitz invariance and Meckes' theorem.

\smallskip
Our second main result concerns the determination of the large-scale behavior of $\Mag(t E_u)$ for the case that all contraction ratios are equal to a constant $r$.
Then $s=-\log_r m$ is the Hausdorff and box dimension of $E$~\cite[Theorem II]{Moran1946}.
Earlier magnitude computations for particular fractals,
including the ternary Cantor set with its Euclidean metric,
were obtained by Leinster and Willerton~\cite[Section 3]{LeinsterWillerton2013}.
Here we obtain a precise asymptotic formula for $\Mag(t E_u)$.

\begin{introthm}\label{introthm:asymptotics}
Suppose that $r_k=r$ for every $k\geq1$.
Set $s=-\log_r m, \ T=-\log r$.
Then there exists a positive smooth $T$-periodic function $\widetilde{p}:\R\to \R_{>0}$ such that
\[
    \Mag(tE_u)
    =\frac{t^s}{\widetilde{p}(\log t)}+o(t^s)
    \qquad
    (t\to\infty).
\]
Moreover,
\[
    \frac{1}{T}\int_0^T\widetilde{p}(x)\,dx
    =\frac{(m-1)\Gamma(s+1)}{m\log m}.
\]
\end{introthm}

Willerton proved that geometric quantities such as volume and total scalar curvature appear in the asymptotic expansion of $\Mag(tX)$ of a homogeneous Riemannian manifold $X$~\cite{Willerton2014}.
More precisely, if $\dim X=n$, then
\[
    \Mag(tX)
    =
    \frac{\Gamma\left(\frac{n}{2}+1\right)\Vol(X)}
         {\Gamma(n+1)\pi^{n/2}}t^n
    +O(t^{n-2})
    \quad
    (t\to\infty).
\]
Here we display only the leading term in~\cite[Theorem 11]{Willerton2014},
since our comparison concerns the highest-order growth.
Thus,
the leading growth is of order $t^n$,
and its coefficient contains geometric information about $X$, namely its volume.

\smallskip
Theorem~\ref{introthm:asymptotics} may be viewed as a fractal counterpart of this phenomenon.
The Hausdorff dimension $s$, which may be nonintegral, replaces the manifold dimension $\dim X$.
In contrast to the Riemannian setting,
the coefficient of $t^s$ is not constant but a log-periodic function $1/\widetilde{p}(x)$.
We consider the reciprocal of the average of $\widetilde{p}(x)$ over one period instead of the coefficient of the leading term in Willerton's formula.
Theorem~\ref{introthm:asymptotics} tells us that the reciprocal of the averaged value explicitly reflects the branching number $m$ and the Hausdorff dimension $s$ through the Gamma function.
In this sense, Theorem~\ref{introthm:asymptotics} gives a Willerton-type leading asymptotic formula for homogeneous Moran ultrametric spaces.

\smallskip
The paper is organized as follows. In Section~2, we review magnitude and magnitude dimensions, construct the homogeneous Moran set $E$, and define the coding ultrametric $d_u$. In Section~3, we approximate $E_u$ by finite homogeneous ultrametric spaces and derive an explicit formula for its magnitude function. In Section~4, we prove Theorem~\ref{introthm:dimensions}, characterize when the Euclidean and coding metrics are bi-Lipschitz equivalent, and prove Theorem~\ref{introthm:asymptotics}.

\section{Preliminaries}
In this section,
we review magnitudes of compact metric spaces,
and construct a Moran set which is homeomorphic to a Cantor set.

\subsection{Magnitude of compact metric spaces}\label{sec:magnitude definition}
Magnitude was first defined for finite metric spaces by Leinster~\cite{Leinster2013}.
Let $X=(X, d)$ be a finite metric space and set $\zeta(x, y)=\exp{(-d(x, y))}$ for $x, y\in X$.
The symmetric matrix $\zeta_X=(\zeta(x, y))_{x,y\in X}\in \R^{X\times X}$ is called the \emph{similarity matrix} of $X$.
A function $\vb*{w}:X\to \R$ is called a \emph{weighting} on $X$ if $\displaystyle \sum_{y\in X}\zeta(x, y)\vb*{w}(y)=1$ for all $x\in X$.
If $\vb*{w}_1$ and $\vb*{w}_2$ are two weightings on $X$,
then we see that 
\begin{align*}
\sum_{x\in X}\vb*{w}_1(x)
&=\sum_{x\in X}\(\sum_{y\in X}\zeta(x, y)\vb*{w}_2(y)\)\vb*{w}_1(x)\\
&=\sum_{y\in X}\(\sum_{x\in X}\zeta(y, x)\vb*{w}_1(x)\)\vb*{w}_2(y)=\sum_{y\in X}\vb*{w}_2(y)\,.
\end{align*}
If $X$ admits a weighting $\vb*{w}$,
then the \emph{magnitude} $\Mag(X)$ of $X$ is defined by
\[
\Mag(X)=\sum_{x\in X}\vb*{w}(x)\,.
\]
Note that not every finite metric space has magnitude.
There are several conditions that guarantee the existence of a weighting.

\smallskip
A metric space is \emph{homogeneous} if its isometry group acts transitively.
In this case,
a weighting can be obtained as follows.
By the transitivity of its isometry group,
all rows of the similarity matrix have the same sum,
say $M$.
Then the constant function $\vb*{w}(x)=M^{-1}$ for $x\in X$ is a weighting on $X$.

\smallskip
Another condition guaranteeing the existence of a weighting is positive definiteness.
A finite metric space $X$ is said to be \emph{positive definite} if the similarity matrix $\zeta_X$ is positive definite.
If so,
since the similarity matrix is invertible,
every positive definite finite metric space admits a unique weighting and hence has magnitude.
Every finite subset of Euclidean space is positive definite.
A metric space $X$ is \emph{ultrametric} if $d(x, z)\leq \max{\{d(x, y), d(y, z)\}}$ for $x,y,z\in X$.
\begin{prop}[{\cite[Proposition 2.4.18]{Leinster2013}}]\label{prop:ultrametic spaces are positive definite}
Every finite ultrametric space is positive definite.
\end{prop}

Meckes extended magnitude from finite metric spaces to compact positive definite metric spaces using a measure-theoretic variational formula~\cite{Meckes2013}.
Let $X=(X,d)$ be a compact metric space, and let $M(X)$ be the Banach space of finite signed Borel measures on $X$ equipped with the total variation norm.
We define a bilinear form $Z_X$ on $M(X)$ by 
\[
Z_X(\mu, \nu)=\int_X\int_X  \exp{(-d(x, y))}\,d\mu(x)\,d\nu(y)\,.
\]
The \emph{magnitude} $\Mag(X)$ of $X$ is defined by
\[
\Mag(X)=\sup{\Set{\dfrac{\mu(X)^2}{Z_X(\mu, \mu)}|\mu\in M(X)\setminus{\{0\}}, \ Z_X(\mu, \mu)\neq 0}}\,.
\]

As in the finite case, homogeneity provides a natural way to construct a weighting measure.
Willerton computed the magnitude of a compact homogeneous Riemannian manifold $X$ using its Riemannian volume measure~\cite{Willerton2014}.
The magnitude is
\[
\Mag(X)=\dfrac{\Vol(X)}{\displaystyle \int_X e^{-d(x, x_0)}\,d\mu(x)},
\]
where $x_0\in X$ is arbitrary,
$\mu$ is the Riemannian volume measure on $X$,
and $\Vol(X)$ is the volume of $X$.

\smallskip
A metric space $X$ is \emph{positive definite} if every finite subset of $X$, equipped with the restricted metric, is positive definite.
Every compact subset of Euclidean space is positive definite.
By Proposition~\ref{prop:ultrametic spaces are positive definite},
every compact ultrametric space is positive definite.
For finite positive definite metric spaces,
the above definition of magnitude coincides with the original definition for finite metric spaces~\cite[Proposition 2.4.3]{Leinster2013}.
We write $d_H$ for the Hausdorff distance between compact subsets of $X$.
\begin{theo}[{\cite[Corollary 2.7]{Meckes2013}}]\label{theo:magnitude approximation}
If $X$ is a compact positive definite metric space and $\{X_n\}$ is a sequence of compact subspaces of $X$ such that $d_H(X_n,X)\to0$ as $n\to\infty$,
then $\Mag(X_n)\to\Mag(X)$ as $n\to\infty$.
\end{theo}

For a compact metric space $X=(X, d)$ and $t>0$,
we denote the $t$-scaled metric space $(X, td)$ by $tX$.
We regard $\Mag(tX)$ as a partially defined function of $t\in \R_{>0}$, defined whenever $tX$ admits a weighting, and call it the \emph{magnitude function} of $X$.
Willerton showed the asymptotic behavior of the magnitude function $\Mag(t X)$ of a compact homogeneous Riemannian manifold $X$ as 
\[
\Mag(t X)=\dfrac{\Gamma(\frac{n}{2}+1)\Vol(X)}{\Gamma(n+1) \pi^{\frac{n}{2}}}\,t^n+o(t^n) \ (t\to \infty)\,,
\]
where $n=\dim X$ (see \cite[Theorem 11]{Willerton2014} for details).
A metric space $X$ is \emph{stably positive definite} if $tX$ is positive definite for every $t>0$.
If $X$ is an ultrametric space,
then so is $tX$ for every $t>0$,
and hence ultrametric spaces are stably positive definite.
For a stably positive definite metric space $X$,
the magnitude function $\Mag(tX)$ is defined as a function of $t\in\R_{>0}$.

\subsection{Box dimensions and magnitude dimensions}
Let $X=(X, d)$ be a compact metric space.
For $t>0$,
the minimum number of closed balls of radius $1/t$ needed to cover $X$ is denoted by $N_{1/t}(X)$.
The \emph{upper box dimension} of $X$ is 
\[
\overline{\dim}_B(X):=\limsup_{t\to \infty}\dfrac{\log{N_{1/t}(X)}}{\log{t}}
\]
and the \emph{lower box dimension} of $X$ is 
\[
\underline{\dim}_B(X):=\liminf_{t\to \infty}\dfrac{\log{N_{1/t}(X)}}{\log{t}}\,.
\]

Meckes introduced the upper and lower magnitude dimensions in~\cite{Meckes2015}.
When $X$ is stably positive definite,
the \emph{upper magnitude dimension} and \emph{lower magnitude dimension} of $X$ are defined by
\[
\overline{\dim}_{\Mag}(X)=\limsup_{t\to \infty}\dfrac{\log \Mag(tX)}{\log t}\,,\quad
\underline{\dim}_{\Mag}(X)=\liminf_{t\to \infty}\dfrac{\log \Mag(tX)}{\log t}\,.
\]

\begin{theo}[{\cite[Corollaries 7.3 and 7.4]{Meckes2015}}\label{theo:Meckes dimensions}]
Let $X=(X,d)$ be a compact metric space.
If $X$ is either an ultrametric space or a subset of Euclidean space equipped with the induced Euclidean metric, then
\[
\overline{\dim}_B(X)=\overline{\dim}_{\Mag}(X)\,,\quad
\underline{\dim}_B(X)=\underline{\dim}_{\Mag}(X)\,.
\]
\end{theo}

The results of Meckes compare box and magnitude dimensions with respect to the same metric,
whereas our result relates dimensions arising from two distinct metric structures on the same underlying set.
To the best of our knowledge,
our result provides the first instance in which the box dimension associated with one metric is recovered as the magnitude dimension associated with a different metric on the same underlying compact set.

\subsection{The homogeneous Moran set $E=E(m, \mathbf{r})$ and the ultrametric $d_u$}\label{sec:Moran set and ultrametric}
In this subsection,
we refer to~\cite{Moran1946} for the original construction of Moran sets.
We denote the length of a bounded interval $J\subset \R$ by $\abs{J}$.
For a finite family $\{J_\ell\}_{\ell=1}^q$ of subintervals contained in an interval $I$,
we call the connected components of
$\displaystyle I\setminus(J_1\cup\dots\cup J_q)$
the \emph{gaps} of $\{J_\ell\}_{\ell=1}^q$ in $I$.

\smallskip
We now define the homogeneous Moran set considered in this paper.
Fix $m\in\N_{\geq2}$ and a sequence
$\mathbf{r}=(r_j)_{j\geq0}$ of positive numbers satisfying
$r_0=1$ and $r_j<m^{-1}$ for $j\geq1$.
Starting from $E_0=[0,1]$,
we construct a decreasing sequence $\{E_k\}_{k\geq 0}$ of compact subsets of $[0,1]$ as follows.
Informally,
to pass from step $k$ to step $k+1$,
each interval constructed at step $k$ is replaced by $m$ mutually disjoint closed subintervals with contraction ratio $r_{k+1}$.
These subintervals are indexed from left to right and are placed so that the
left endpoint of the parent interval coincides with the left endpoint of the
leftmost subinterval, the right endpoint of the parent interval coincides with
the right endpoint of the rightmost subinterval, and the $m-1$ gaps between
consecutive subintervals have equal length.

\smallskip
To make this construction precise, let $[m]=\{1,\ldots,m\}$ and let $[m]^k$
be the $k$-fold Cartesian product of $[m]$, that is, $[m]^k=\Set{\mathbf{i}=(i_1,\ldots,i_k)|1\leq i_1,\ldots,i_k\leq m}$.
For $k=0$,
we put $[m]^0=\{\emptyset\}$,
where $\emptyset$ denotes the empty word,
and set $I_{\emptyset}=[0,1]$.
For $\mathbf{i}=(i_1,\ldots,i_\ell)\in[m]^\ell$ and $j\in[m]$,
we denote by
\[
(\mathbf{i},j)=(i_1,\ldots,i_\ell,j)\in[m]^{\ell+1}
\]
the concatenation of $\mathbf{i}$ and $j$.
When $\mathbf{i}=\emptyset$, we adopt the convention
\[
(\emptyset,j)=(j).
\]

Fix $k\in\Z_{\geq0}$ and suppose that a family $\Set{I_{\mathbf{i}}|\mathbf{i}\in[m]^\ell,\ 0\leq\ell\leq k}$ of closed intervals has been constructed so that, for each
$0\leq\ell\leq k$,
\[
E_\ell
=
\bigcup_{\mathbf{i}\in[m]^\ell}I_{\mathbf{i}},
\quad
\abs{I_{\mathbf{i}}}
=
r_0\cdots r_\ell
\quad
\text{for every }\mathbf{i}\in[m]^\ell.
\]
Moreover, for every $0\leq\ell\leq k-1$ and every
$\mathbf{i}\in[m]^\ell$,
the intervals
\[
I_{(\mathbf{i},1)},\ldots,I_{(\mathbf{i},m)}
\]
are indexed from left to right,
the left endpoint of $I_{\mathbf{i}}$ coincides with the left endpoint of
$I_{(\mathbf{i},1)}$,
the right endpoint of $I_{\mathbf{i}}$ coincides with the right endpoint of
$I_{(\mathbf{i},m)}$,
and the gaps between consecutive intervals
$I_{(\mathbf{i},j)}$ and $I_{(\mathbf{i},j+1)}$
have equal length for $1\leq j\leq m-1$.

\smallskip
For each $\mathbf{i}\in[m]^k$,
choose $m$ mutually disjoint closed subintervals
\[
I_{(\mathbf{i},1)},\ldots,I_{(\mathbf{i},m)}
\subset I_{\mathbf{i}},
\]
indexed from left to right, such that
\[
\abs{I_{(\mathbf{i},j)}}
=
r_0\cdots r_{k+1}
\qquad
(1\leq j\leq m),
\]
the left endpoint of $I_{\mathbf{i}}$ coincides with the left endpoint of
$I_{(\mathbf{i},1)}$,
the right endpoint of $I_{\mathbf{i}}$ coincides with the right endpoint of
$I_{(\mathbf{i},m)}$,
and the gaps between consecutive intervals
$I_{(\mathbf{i},j)}$ and $I_{(\mathbf{i},j+1)}$
have equal length for $1\leq j\leq m-1$.
Each of these gaps has length
\[
\frac{r_0\cdots r_k-mr_0\cdots r_{k+1}}{m-1}
=
\frac{r_0\cdots r_k(1-mr_{k+1})}{m-1}.
\]
We then define
\[
E_{k+1}
=
\bigcup_{\mathbf{i}\in[m]^{k+1}}I_{\mathbf{i}}.
\]
Finally, the associated homogeneous Moran set is
\[
E
=
\bigcap_{k=0}^{\infty}E_k.
\]
We write $E(m,\mathbf{r})$ when the dependence on $m$ and $\mathbf{r}$ needs
to be emphasized, and simply write $E$ otherwise.
The $m^k$ intervals $I_{\mathbf{i}}$, $\mathbf{i}\in[m]^k$, are called the
intervals of \emph{step} $k$.

\smallskip
We define an ultrametric on the homogeneous Moran set $E=E(m,\mathbf{r})$ via its coding.
Let $\Omega_m$ denote the set of sequences with entries in $[m]$,
that is,
\[
\Omega_m=\Set{\mathbf{i}=(i_k)_{k\in \N}|i_k\in[m]\text{ for every }k\in\N}.
\]
Define
\[
R_k=r_0\cdots r_k\quad (k\in\mathbb{Z}_{\geq0}).
\]
For $\mathbf{i}, \mathbf{j}\in \Omega_m$,
set
\[
\presentation{\mathbf{i}}{\mathbf{j}}=
\begin{cases}
0 & (i_1\neq j_1) \\
\infty & (\mathbf{i}=\mathbf{j}) \\
\max{\Set{k\in \N|i_\ell=j_\ell\text{ for }1\leq\ell\leq k}} & \text{otherwise}
\end{cases}
\]
Thus, $\presentation{\mathbf{i}}{\mathbf{j}}$ is the length of the longest common prefix of $\mathbf{i}$ and $\mathbf{j}$, with the convention that it is $\infty$ when $\mathbf{i}=\mathbf{j}$.
Then define
\[
d_u(\mathbf{i}, \mathbf{j})=
\begin{cases}
0 & (\mathbf{i}=\mathbf{j}) \\
R_{\presentation{\mathbf{i}}{\mathbf{j}}} & \text{otherwise}
\end{cases}.
\]
It is easy to see that $d_u$ is an ultrametric on $\Omega_m$ that induces the product topology.
In particular,
the ultrametric space $(\Omega_m, d_u)$ is compact.

\smallskip
Since the intervals of each step are mutually disjoint,
for every $k\in\N$, each point $x\in E$ is contained in a unique interval of step $k$.
We denote by $\mathbf{i}_k(x)\in[m]^k$ the index of the interval of step $k$ containing $x$.
The indices $\mathbf{i}_k(x)$ are compatible in the following sense.
For every $k\in\N$, there exists $i_{k+1}\in[m]$ such that
\[
\mathbf{i}_{k+1}(x)=(\mathbf{i}_k(x),i_{k+1}).
\]
Conversely, for every $\mathbf{i}=(i_k)_{k\in\N}\in\Omega_m$,
the intervals
\[
I_{(i_1)}
\supset I_{(i_1,i_2)}
\supset\cdots
\]
form a decreasing sequence of nonempty compact intervals whose lengths tend to zero.
Hence their intersection consists of a single point of $E$.
Thus, we obtain a bijection $\sigma:E\to\Omega_m$,
called the \emph{coding},
such that if $\sigma(x)=(i_k)_{k\in \N}$,
then $\mathbf{i}_k(x)=(i_1,\dots,i_k)$ for $k\in \N$.
Via the bijection $\sigma$,
we pull back $d_u$ to an ultrametric on $E$, which we also denote by $d_u$.
The compact ultrametric space $(E, d_u)$ is denoted by $E_u$.

\section{The magnitude function of the compact ultrametric space $E_u$}
In this section,
we compute the magnitude function of the compact ultrametric space $E_u$ defined in Section~\ref{sec:Moran set and ultrametric}.
For each $n\geq 1$ and $i_1,\dots,i_n\in [m]$,
choose a point $x_{i_1,\dots,i_n}\in E$ such that the first $n$ coordinates of
$\sigma(x_{i_1,\dots,i_n})$ are $i_1,\dots,i_n$.
The set of these points is denoted by $E_n$,
that is,
\begin{equation*}
E_n=\Set{x_{i_1,\dots,i_n}|i_1,\dots,i_n\in[m]}\,.
\end{equation*}
We denote by $E_{u,n}$ the set $E_n$ equipped with the metric induced from $E_u$.

\begin{lem}\label{lem:approximation of moran set}
The sequence $\{E_{u, n}\}_{n\in \N}$ converges to $E_u$ in the Hausdorff distance.
\end{lem}
\begin{proof}
We write $d_H$ for the Hausdorff distance and show that $d_H(E_{u, n}, E_u)\leq R_n$ for $n\in \N$.
Fix an arbitrary point $y\in E$ and set $\sigma(y)=(i_k)_{k\in \N}$.
Since $\sigma(y)$ and $\sigma(x_{i_1,\dots,i_n})$ have a common prefix of length at least $n$,
we have $\presentation{\sigma(y)}{\sigma(x_{i_1,\dots,i_n})}\geq n$,
and hence $d_u(y, x_{i_1,\dots,i_n})\leq R_n$.
Since $R_n\to0$ as $n\to\infty$,
the assertion follows.
\end{proof}

\begin{lem}\label{lem:magnitude of n step moran set}
All row sums of the similarity matrix $\zeta_{tE_{u,n}}$ are equal.
In particular, the magnitude function $\Mag(tE_{u,n})$ is given by
\[
\Mag(tE_{u, n})=\dfrac{m^n}{\displaystyle \sum_{k=0}^{n-1}\(1-\dfrac{1}{m}\)m^{n-k}\exp{(-tR_k)}+1}\,.
\]
\end{lem}
\begin{proof}
Fix $x=x_{i_1,\dots,i_n}\in E_{u,n}$ and denote by $S_n(t,x)$ the row sum of $\zeta_{tE_{u,n}}$ corresponding to $x$,
that is,
\begin{equation*}
S_n(t, x)=\sum_{y\in E_{u,n}}\exp{(-td_u(x, y))}.
\end{equation*}
For $k\in{\{0,\dots,n\}}$,
we define a subset $E_{u,n}(x, k)$ of $E_{u,n}$ by 
\begin{align*}
E_{u,n}(x, 0)&=\Set{y\in E_{u,n}|j_1\neq i_1},\\
E_{u,n}(x, k)&=\Set{y\in E_{u,n}|j_\ell=i_\ell\text{ for }1\leq\ell\leq k, \ j_{k+1}\neq i_{k+1}} \ (1\leq k\leq n-1),\\
E_{u,n}(x, n)&=\Set{y\in E_{u,n}|j_1=i_1,\dots,j_n=i_n},
\end{align*}
where $\sigma(y)=(j_\ell)_{\ell\in \N}$ for $y\in E_{u, n}$.
It follows that $E_{u,n}(x, n)=\{x\}$ and $E_{u,n}=\displaystyle \bigsqcup_{k=0}^n E_{u,n}(x, k)$,
and hence
\[
S_n(t, x)=\sum_{k=0}^{n-1}\sum_{y\in E_{u,n}(x, k)}\exp{(-t d_u(x, y))}+1\,.
\]
Since $d_u(x, y)=R_k$ for  $0\leq k\leq n-1$ and $y\in E_{u, n}(x, k)$,
we have
\[
S_n(t, x)=\sum_{k=0}^{n-1}\#E_{u,n}(x, k)\exp{(-t R_k)}+1\,.
\]
For $0\leq k\leq n-1$ and $y\in E_{u,n}(x, k)$ with $\sigma(y)=(j_\ell)_{\ell\in \N}$,
there are $m-1$ choices for $j_{k+1}$, while each of $j_{k+2},\dots,j_n$ has $m$ choices.
Therefore,
\[
\#E_{u,n}(x,k)=(m-1)m^{n-k-1}=\(1-\dfrac{1}{m}\)m^{n-k},
\]
and hence
\[
S_n(t, x)=\sum_{k=0}^{n-1}\(1-\dfrac{1}{m}\)m^{n-k}\exp{(-t R_k)}+1\,.
\]
In particular,
all row sums of $\zeta_{tE_{u,n}}$ are equal.
It follows that the function $\vb*{w}:E_{u,n}\to\R$ defined by
\[
\vb*{w}(x)=S_n(t,x)^{-1}
\]
is a weighting on $E_{u,n}$.
Since $\#E_{u,n}=m^n$,
the assertion follows.
\end{proof}

\begin{prop}\label{prop:magnitude function of Eu}
The magnitude function of $E_u$ is given by
\[
\Mag(t E_u)=\dfrac{1}{\,\displaystyle \sum_{k=0}^\infty \(1-\dfrac{1}{m}\)m^{-k}\exp{(-t R_k)}\,}\,.
\]
\end{prop}
\begin{proof}
By Lemmas~\ref{lem:approximation of moran set} and~\ref{lem:magnitude of n step moran set},
together with Theorem~\ref{theo:magnitude approximation},
\begin{align*}
\Mag(tE_u)
&=\lim_{n\to \infty}\Mag(tE_{u,n})\\
&=\lim_{n\to \infty}\dfrac{m^n}{\displaystyle \sum_{k=0}^{n-1}\(1-\dfrac{1}{m}\)m^{n-k}\exp{(-tR_k)}+1}\\
&=\lim_{n\to \infty}\dfrac{1}{\displaystyle \sum_{k=0}^{n-1}\(1-\dfrac{1}{m}\)m^{-k}\exp{(-tR_k)}+m^{-n}}\\
&=\dfrac{1}{\,\displaystyle \sum_{k=0}^\infty \(1-\dfrac{1}{m}\)m^{-k}\exp{(-t R_k)}\,}\,.
\end{align*}
\end{proof}

\section{Proofs of Theorems~\ref{introthm:dimensions} and~\ref{introthm:asymptotics}}
\begin{theo}\label{theo:main1}
For the homogeneous Moran set $E=E(m,\mathbf{r})$ defined in Section~\ref{sec:Moran set and ultrametric},
\[
\overline{\dim}_{\Mag}(E_u)=\overline{\dim}_B(E)\,,\quad
\underline{\dim}_{\Mag}(E_u)=\underline{\dim}_B(E)\,.
\]
\end{theo}
\begin{proof}
By Proposition~\ref{prop:magnitude function of Eu},
we denote the denominator of $\Mag(tE_u)$ by $D(t)$,
that is,
\[
D(t)=\sum_{k=0}^\infty \(1-\dfrac{1}{m}\)m^{-k}\exp{(-t R_k)}\,.
\]
For $t>1$, set $k(t)=\min\Set{k\in\N| R_k\leq \dfrac{1}{t}}$.
It follows that
\[
R_{k(t)}\leq \dfrac{1}{t}<R_{k(t)-1}\,.
\]

First,
we show that there are positive constants $c>0$ and $C>0$ such that
\begin{equation}
c\,m^{-k(t)}\leq D(t)\leq C\,m^{-k(t)}\,.\label{eq:estimates of D}
\end{equation}
For the lower bound,
\begin{equation*}
D(t)\geq \(1-\dfrac{1}{m}\)m^{-k(t)}\exp{(-t R_{k(t)})}\geq \(1-\dfrac{1}{m}\)e^{-1}\,m^{-k(t)}\,.
\end{equation*}
For the upper bound,
\begin{align}
D(t)
&=\sum_{k=0}^{k(t)-1}\(1-\dfrac{1}{m}\)m^{-k}\exp{(-t R_k)}+\sum_{k=k(t)}^\infty \(1-\dfrac{1}{m}\)m^{-k}\exp{(-t R_k)} \notag\\
&\leq \sum_{k=0}^{k(t)-1}\(1-\dfrac{1}{m}\)m^{-k}\exp{(-t R_k)}+\sum_{k=k(t)}^\infty \(1-\dfrac{1}{m}\)m^{-k} \notag\\
&\leq \sum_{k=0}^{k(t)-1}\(1-\dfrac{1}{m}\)m^{-k}\exp{(-t R_k)}+m^{-k(t)}\,.\label{eq:estimate of D from above}
\end{align}
In order to give an estimate of the first term of \eqref{eq:estimate of D from above} from above,
we derive a lower bound for $tR_k$.
For $0\leq k\leq k(t)-1$,
\[
R_k=r_1\dots r_k=\dfrac{R_{k(t)-1}}{r_{k+1}\dots r_{k(t)-1}}\geq \dfrac{R_{k(t)-1}}{m^{-(k(t)-1-k)}}\geq \dfrac{1}{t}\,m^{k(t)-(k+1)},
\]
and hence
\[
tR_k\geq m^{k(t)-(k+1)}\,.
\]
Therefore,
we see from \eqref{eq:estimate of D from above} that
\begin{align*}
D(t)
&\leq \sum_{k=0}^{k(t)-1}\(1-\dfrac{1}{m}\)m^{-k}\exp{(-m^{k(t)-1-k})}+m^{-k(t)}\\
&=\sum_{k=0}^{k(t)-1}\(1-\dfrac{1}{m}\)m^{k-(k(t)-1)}\exp{(-m^k)}+m^{-k(t)}\\
&=\(\sum_{k=0}^{k(t)-1}\(m-1\)m^{k}\exp{(-m^k)}+1\)m^{-k(t)}\\
&\leq \(\sum_{k=0}^{\infty}\(m-1\)m^{k}\exp{(-m^k)}+1\)m^{-k(t)}\,,
\end{align*}
where the series on the right-hand side of the last line converges.

\smallskip
Next, 
we establish the following bounds for $N_{1/t}(E)$.
\begin{equation}
m^{k(t)-2}\leq N_{1/t}(E)\leq m^{k(t)}\,.\label{eq:estimates of N}
\end{equation}
Since $E_{k(t)}$ covers $E$ and consists of $m^{k(t)}$ intervals of length $R_{k(t)}\leq 1/t$,
we obtain the upper bound.
To prove the lower bound,
set $N=N_{1/t}(E)$ and take closed intervals $J_1,\dots,J_N$ of lengths at most $2/t$ that cover $E$.
We show that each $J_i$ contains at most $2m^2$ endpoints of intervals in $E_{k(t)}$.
Suppose, to the contrary, that $J_i$ contains at least $2m^2+1$ endpoints of intervals in $E_{k(t)}$.
Since $J_i$ is an interval,
we may choose $2m^2+1$ consecutive endpoints contained in $J_i$.
It follows that $J_i$ contains at least $m^2$ consecutive intervals of step $k(t)$ and at least $m^2$ gaps between consecutive such intervals.
Let $g_{k(t)}$ denote the common length of the gaps created at step $k(t)$.
By the construction of $E$,
\[
g_{k(t)}=\dfrac{R_{k(t)-1}-mR_{k(t)}}{m-1}\,.
\]
Enumerate the gaps in $E_{k(t)}$ from left to right as $\Gamma_1,\Gamma_2,\dots,\Gamma_{m^{k(t)}-1}$.
Then, for $1\leq j\leq m^{k(t)}-1$,
$\Gamma_j$ is also a gap in $E_{k(t)-1}$ if and only if $j\in m\Z$.
This observation tells us that there are $m^2-m$ gaps of length $g_{k(t)}$ among any $m^2$ consecutive gaps in $E_{k(t)}$.
Therefore,
the interval $J_i$ contains at least $m^2-m$ gaps of length $g_{k(t)}$,
so
\[
\abs{J_i}>m^2R_{k(t)}+(m^2-m)g_{k(t)}=mR_{k(t)-1}\,.
\]
Since $\abs{J_i}\leq \dfrac{2}{t}<2R_{k(t)-1}\leq mR_{k(t)-1}$,
we obtain a contradiction,
so that each covering interval $J_i$ contains at most $2m^2$ endpoints in $E_{k(t)}$.
All endpoints of the intervals in $E_{k(t)}$ are covered by $J_1,\dots,J_N$,
so $2m^{k(t)}\leq N\cdot 2m^2$,
and hence we obtain the estimate of $N_{1/t}(E)$ from below.

By the inequalities \eqref{eq:estimates of D} and \eqref{eq:estimates of N},
there are positive constants $A, B>0$ such that 
\[
A \Mag(tE_u)\leq N_{1/t}(E)\leq B \Mag(t E_u),
\]
and hence the assertion follows.
\end{proof}
Since the upper and lower box dimensions of compact metric spaces are invariant under bi-Lipschitz equivalence,
if $E$ and $E_u$ are bi-Lipschitz equivalent, then Theorem~\ref{theo:main1} is an immediate consequence of Theorem~\ref{theo:Meckes dimensions}.
To clarify the difference between our result and Meckes' result,
we characterize when the ultrametric $d_u$ is bi-Lipschitz equivalent to the Euclidean metric $d$.

\begin{prop}\label{prop:difference of metrics}
The ultrametric $d_u$ and the Euclidean metric $d$ on $E$ are bi-Lipschitz equivalent if and only if $\displaystyle \sup_{k\geq 1} r_k < m^{-1}$\,.
\end{prop}
\begin{proof}
First we show that if $\displaystyle \sup_{k\geq 1} r_k<m^{-1}$,
then $d_u$ and $d$ are bi-Lipschitz equivalent.
Set $\displaystyle c=\inf_{k\geq1}\frac{1-mr_k}{m-1}>0$.
We show that
\[
c\,d_u(x,y)\leq d(x,y)\leq d_u(x,y)
\]
for all $x,y\in E$.
Fix distinct points $x,y\in E$.
Suppose that $d_u(x, y)=R_{k-1}$.
It follows that $x$ and $y$ lie in the same interval of step $k-1$ and in two distinct intervals of step $k$.
Since every interval in $E_{k-1}$ has length $R_{k-1}$,
\[
d(x,y)\leq R_{k-1}=d_u(x,y).
\]
Since any gap between consecutive intervals of step $k$ contained in the same interval of step $k-1$ has length $\dfrac{R_{k-1}-mR_k}{m-1}$,
\[
d(x, y)\geq \dfrac{1}{m-1}(R_{k-1}-mR_k)=\dfrac{1-mr_k}{m-1}d_u(x, y)\geq c\,d_u(x, y)\,.
\]

Conversely, suppose that $d$ and $d_u$ are bi-Lipschitz equivalent.
Then there exists $C>0$ such that $d(x,y)\geq C d_u(x,y)$ for all $x,y\in E$.
For each $k\in \N$,
choose two consecutive intervals of step $k$ contained in the same interval of step $k-1$ and consider the gap between these two intervals.
We write $x_k, y_k$ for the endpoints of the gap.
It follows from the construction of $E$ that $x_k, y_k\in E$,
and hence
\[
C\leq \dfrac{d(x_k, y_k)}{d_u(x_k, y_k)}=\dfrac{\frac{1}{m-1}(R_{k-1}-mR_k)}{R_{k-1}}=\dfrac{1-mr_k}{m-1}\,.
\]
Therefore $\displaystyle \inf_{k\geq 1}\dfrac{1-mr_k}{m-1}>0$ or, equivalently $\displaystyle \sup_{k\geq 1}r_k<\frac{1}{m}$.
\end{proof}

\begin{theo}
Suppose that $r_j=r$ for every $j\geq1$.
Set $s=-\log_r m$ and $T=-\log{r}$.
Then,
there exists a positive smooth $T$-periodic function $\widetilde{p}:\R\to\R_{>0}$ such that 
\begin{itemize}
\item[(1)]
$\displaystyle \Mag(t E_u)=\dfrac{t^s}{\,\widetilde{p}(\log t)\,}+o(t^s) \ (t\to \infty)$.
\item[(2)]
$\dfrac{1}{T}\displaystyle \int_0^T \widetilde{p}(x)\,dx=\dfrac{(m-1)\Gamma(s+1)}{\,m\log m\,}$\,.
\end{itemize}
\end{theo}
\begin{proof}
The proof is similar to the computation of the magnitude of ternary Cantor sets in~\cite{LeinsterWillerton2013}.
We define two functions $p(t), q(t)$ on $\R_{>0}$ as follows.
\[
p(t)=\(1-\dfrac{1}{m}\)\sum_{k=-\infty}^\infty m^{-k}\exp(-tr^k)\,,\quad
q(t)=\(1-\dfrac{1}{m}\)\sum_{k=1}^\infty m^k\exp(-tr^{-k})\,.
\]
It follows from Proposition~\ref{prop:magnitude function of Eu} that
\[
\Mag(tE_u)=\frac{1}{p(t)-q(t)}.
\]
Since $p(r^{-1}t)=\dfrac{1}{m}p(t)$,
we see that
\[
(r^{-1}t)^s\,p(r^{-1}t)
=
t^s\cdot \dfrac{r^{-s}}{m}\cdot p(t)
=t^s\,p(t)\,.
\]
The defining series for $p(t)$ and all of its termwise derivatives converge locally uniformly on $\R_{>0}$ by the Weierstrass $M$-test.
Hence $p$ is smooth on $\R_{>0}$.
We define a smooth positive function $\widetilde{p}(x) \ (x\in \R)$ as follows.
\[
\widetilde{p}(x)=e^{sx}p(e^x)\,.
\]
It follows that $\widetilde{p}(x)$ has period $T=-\log{r}$.

Next,
we show that $\disp \lim_{t\to \infty}t^s q(t)=0$.
For $k\in \N$,
\[
\dfrac{m^{k+1}\exp(-tr^{-(k+1)})}{m^k\exp(-tr^{-k})}=m\exp(-tr^{-k}(r^{-1}-1))\leq m\exp(-t(r^{-1}-1))\,.
\]
Since $\displaystyle \lim_{t\to \infty} \exp(-t(r^{-1}-1))=0$,
for sufficiently large $t>0$,
\[
\dfrac{m^{k+1}\exp(-tr^{-(k+1)})}{m^k\exp(-tr^{-k})}\leq \dfrac{1}{2},
\]
and hence
\begin{align*}
\lim_{t\to \infty}t^sq(t)
&\leq \lim_{t\to \infty}t^s\(1-\dfrac{1}{m}\)\sum_{k=1}^\infty \dfrac{1}{2^{k-1}}\cdot m\exp(-tr^{-1})\\
&=2(m-1)\lim_{t\to \infty}t^s\exp(-tr^{-1})\\
&=0\,.
\end{align*}
Since $\widetilde{p}$ is positive and $T$-periodic,
there exists a constant $c_1>0$ such that $\widetilde{p}(x)\geq c_1$ for all $x\in\R$.
Moreover,
since $\displaystyle \lim_{t\to\infty}t^sq(t)=0$,
we have $\displaystyle \widetilde{p}(\log t)-t^sq(t)\geq \frac{c_1}{2}$ for sufficiently large $t>0$.
Therefore,
for sufficiently large $t>0$,
\begin{align*}
\dfrac{1}{t^s}
\abs{
\Mag(tE_u)
-
\dfrac{t^s}{\widetilde{p}(\log t)}
}
&=
\abs{
\dfrac{1}{t^sp(t)-t^sq(t)}
-
\dfrac{1}{\widetilde{p}(\log t)}
}\\
&=
\abs{
\dfrac{1}{\widetilde{p}(\log t)-t^sq(t)}
-
\dfrac{1}{\widetilde{p}(\log t)}
}\\
&=
\dfrac{t^sq(t)}
{\widetilde{p}(\log t)
\abs{\widetilde{p}(\log t)-t^sq(t)}}\\
&\leq
\dfrac{2t^sq(t)}{c_1^2}.
\end{align*}
Since $\displaystyle \lim_{t\to\infty}t^sq(t)=0$,
it follows that
\[
\Mag(tE_u)
=
\dfrac{t^s}{\widetilde{p}(\log t)}
+
o(t^s).
\]

\smallskip
Finally,
we show that $\displaystyle \dfrac{1}{T}\int_0^T \widetilde{p}(x)\,dx=\dfrac{(m-1)\Gamma(s+1)}{m\log m}$.
Since
\begin{align*}
\widetilde{p}(x)
&=e^{sx}\(1-\dfrac{1}{m}\)\sum_{k=-\infty}^\infty m^{-k}\exp(-e^xr^k)\\
&=e^{sx}\(1-\dfrac{1}{m}\)\sum_{k=-\infty}^\infty e^{-k\log{m}}\exp(-e^x\cdot e^{k\log{r}})\\
&=\(1-\dfrac{1}{m}\)\sum_{k=-\infty}^\infty e^{s(x-k\,\frac{\,\log{m}\,}{s})}\exp(-e^{x-k(-\log{r})})\\
&=\(1-\dfrac{1}{m}\)\sum_{k=-\infty}^\infty e^{s(x-kT)}\exp(-e^{x-kT}),
\end{align*}
we have
\begin{align*}
\dfrac{1}{T}\int_0^T \widetilde{p}(x)\,dx
&=\dfrac{1}{T}\(1-\dfrac{1}{m}\)\sum_{k=-\infty}^\infty \int_0^T e^{s(x-kT)}\exp(-e^{x-kT})\,dx\\
&=\dfrac{1}{T}\(1-\dfrac{1}{m}\)\sum_{k=-\infty}^\infty \int_{-kT}^{-(k-1)T} e^{su}\exp(-e^u)\,du\\
&=\dfrac{1}{T}\(1-\dfrac{1}{m}\)\int_{-\infty}^{\infty} e^{su}\exp(-e^u)\,du\\
&=\dfrac{1}{T}\(1-\dfrac{1}{m}\)\int_{0}^{\infty} y^se^{-y}\,\dfrac{1}{\,y\,}dy\\
&=\dfrac{1}{T}\(1-\dfrac{1}{m}\)\Gamma(s)
=\dfrac{\Gamma(s+1)}{sT}\(1-\dfrac{1}{m}\)
=\dfrac{(m-1)\Gamma(s+1)}{m\log m}.
\end{align*}
\end{proof}

\section*{Acknowledgement}
The authors would like to thank Takayuki Watanabe for pointing out that the sets considered in this paper are known as homogeneous Moran sets in the literature.
The second author was partially supported by JSPS Grant-in-Aid for Early-Career Scientists Grant Number JP25K17258.

\bibliographystyle{plain}
\bibliography{reference}
\end{document}